\documentclass[arXiv]{actapoly}
\pdfoutput=1

\usepackage[english]{babel}
\usepackage[utf8]{inputenc}

\usepackage{color}

\usepackage{enumitem}
\setlist[enumerate]{label=\roman*), ref=\theenumii.\roman*}

\newcommand{\Zomega}{\mathbb{Z}[\omega]}
\newcommand{\Zbeta}{\mathbb{Z}[\beta]}

\newcommand{\ZZ}{\mathbb{Z}}
\newcommand{\QQ}{\mathbb{Q}}
\newcommand{\CC}{\mathbb{C}}
\newcommand{\NN}{\mathbb{N}}
\newcommand{\RR}{\mathbb{R}}

\newcommand{\fin}[2][\beta]{\operatorname{Fin}_{#2}(#1)}

\newcommand{\A}{\mathcal{A}}
\newcommand{\B}{\mathcal{B}}

\newcommand{\vect}[1]{({#1}_0,{#1}_1,\cdots,{#1}_{d-1})^T}
\newcommand{\multMat}[1]{\sum_{i=0}^{d-1} {#1}_i S^i}

\newtheorem{theorem}{Theorem}[section]
\newtheorem{lemma}[theorem]{Lemma}
\newtheorem{corollary}[theorem]{Corollary}

\theoremstyle{definition}
\newtheorem{definition}[theorem]{Definition}

\renewcommand\Re[1]{\operatorname{Re}#1}

\begin{document}

\title{Minimal non-integer alphabets allowing parallel addition}

\correspondingauthor[J. Legersk\'y]{Jan Legersk\'y}{risc,fjfi}{jan.legersky@risc.jku.at}

\institution{risc}{Research Institute for Symbolic Computation, Johannes Kepler University \\ Altenbergerstra{\ss}e 69, A-4040 Linz, Austria}
\institution{fjfi}{Faculty of Nuclear Sciences and Physical Engineering, Czech Technical University in Prague\\
Trojanova 13, 120 00 Praha 2, Czech Republic}

\begin{abstract}
Parallel addition, i.e., addition with limited carry propagation has been so far studied for complex bases and integer alphabets.
We focus on alphabets consisting of integer combinations of powers of the base.
We give necessary conditions on the alphabet allowing parallel addition.
Under certain assumptions, we prove the same lower bound on the size of the generalized alphabet that is known for alphabets consisting of consecutive integers.
We also extend the characterization of bases allowing parallel addition to numeration systems with non-integer alphabets.
\end{abstract}


\maketitle

% %%%%%%%%%%%%%%%%%%%%%%%%%%%%%%%%%%%%%%%%%%%%%%%%%%%%%%%%%%%%%%%%%%%%%%%%%%%%%%%%%%%%%%%%%%%%%%%%%%
\section{Introduction}
The concept of parallel addition in a numeration system with a base $\beta$ and alphabet~$\A$ was introduced by A. Avizienis \cite{avizienis}.
The crucial difference from standard addition is that carry propagation is limited and hence an output digit depends only on bounded number of input digits.
Therefore, the whole operation can run in constant time in parallel.

It is know that the alphabet $\A$ must be redundant~\cite{kornerup}, otherwise parallel addition is not possible.
Necessary conditions on the base and alphabet were further studied by C. Frougny, P. Heller, E. Pelantov\'a, and M. Svobodov\'a \cite{kBlock,minAlph,parAddNS} under assumption that the alphabet $\A$ consists of consecutive integers containing~0.
It was shown that there exists an integer alphabet allowing parallel addition if and only if the base is an algebraic number with no conjugates of modulus 1.
Lower bounds on the size of the alphabet were given.

The main result of this paper is generalization of these results to non-integer alphabets, namely $\A\subset\Zbeta$. 
Such alphabets might have elements smaller in modulus comparing to integer ones.
This is useful for instance in online multiplication and division~\cite{Brzicova2016}.
Parallel addition algorithms that use non-integer alphabets are discussed in~\cite{constParAdd}.
The paper \cite{Baker2017} discusses consequences of parallel addition for eventually periodic representations in $\QQ(\beta)$.

This paper is organized as follows: in Section~\ref{sec:preliminaries}, we recall the necessary definitions and show that for parallel addition we can consider only bases being algebraic numbers.
In Section~\ref{sec:necessaryAlphabet}, we prove that if $(\beta, \A)$ allows parallel addition and $\beta'$ is a conjugate of $\beta$,
then there is an alphabet $\A'$ such that $(\beta', \A')$ allows parallel addition. 
If $\A[\beta]=\Zbeta$, we show that $\A$ must contain all representatives modulo $\beta$ and $\beta-1$.
If $\beta$ is an algebraic integer, a consequence is the same lower bound on the size of $\A\subset\Zbeta$ as for integer alphabets.
The assumption $\A[\beta]=\Zbeta$ or existence of parallel addition without anticipation implies that $\beta$ is expanding, i.e., all its conjugates are greater than one in modulus.

The key result from \cite{kBlock} is generalized to $\A\subset\Zbeta$ in Section~\ref{sec:condBase}. 
Namely, there is an alphabet in $\Zbeta$ allowing so-called \emph{$k$-block} parallel addition if and only if $\beta$ is an algebraic number with no conjugates of modulus one.
%%%%%%%%%%%%%%%%%%%%%%%%%%%%%%%%%%%%%%%%%%%%%%%%%%%%%%%%%%%%%%%%%%%%%%%%%%%%%%%%%%%%%%%%%%%%%%%%%%%
\section{Preliminaries}
\label{sec:preliminaries}
The concept of positional numeration systems with integer bases and digits is very old and can be easily generalized: 
\begin{definition}
	If $\beta \in \CC$ is such that $|\beta|>1$ and $\A \subset \CC$ is a finite set containing 0,
	then the pair $(\beta, \A)$ is called a \emph{numeration system} with a \emph{base} $\beta$ and \emph{digit set} $\A$,
	usually called an \emph{alphabet}.
\end{definition}

Numbers in a numeration system $(\beta, \A)$ are represented in the following way: let $x$ be a complex number and $x_n,x_{n-1}, x_{n-2},\ldots \in\A, n\geq 0$. We say that $^\omega0 x_n x_{n-1}\cdots x_1 x_0 \bullet x_{-1} x_{-2} \cdots $ is a \emph{$(\beta, \A)$-representation} of $x$ if~$x=\sum_{j=-\infty}^n x_j \beta^j$, where $^\omega0$  denotes the left-infinite sequence of zeros.

The set of all numbers which have a  $(\beta,\A)$-representation with only finitely many non-zero digits is denoted by
$$
\fin{\A}:=\left\{\sum_{j=-m}^n x_j \beta^j\colon n, m \in \NN, x_j \in \A \right\}\,.
$$
The set of all numbers with a finite $(\beta,\A)$-representation with only non-negative powers of $\beta$ is denoted by
$$
\A[\beta]:=\left\{\sum_{j=0}^n x_j \beta^j\colon n\in \NN, x_j \in \A \right\}\,.
$$

We remark that the definition of $\A[\beta]$ is analogous to the one of  $\ZZ[\beta]$, i.e. the smallest ring containing $\ZZ$ and $\beta$, which is equivalent to the set of all sums of powers of $\beta$ with integer coefficients.

Now we show that whenever we require the alphabet to be finite and the  sum of two numbers with finite $(\beta, \A)$-representations to have again a finite $(\beta, \A)$-representation (which is the case of parallel addition), then we can consider only bases which are algebraic numbers.
\begin{lemma}
\label{lem:betaMustBeAlgebraic}
	Let $\beta$ be a complex number such that $|\beta|>1$ and $\A\subset \ZZ[\beta]$ be a finite alphabet 
	with $0\in\A$ and  $1\in\fin{\A}$. 
	If $\NN\subset \fin{\A}$, then $\beta$ is an algebraic number.
\end{lemma}
\begin{proof}
Since $\A\subset \Zbeta$, all digits can be expressed as finite integer combinations of powers of $\beta$. 
Let $d$ be the maximal exponent of $\beta$ occurring in these expressions and 
$C$ be the maximal absolute value of the integer coefficients of all digits in $\A$.

Hence, for every $N\in\NN$, there exist $m, n\in \NN$ and $a_{-m}, \dots, a_{n}\in\A$, 
where $a_i=\sum_{j=0}^d \alpha_{ij}\beta^j$ with $\alpha_{ij}\in\ZZ$ and $|\alpha_{ij}| \leq C$, such that
\begin{align*}
N=\sum_{i=-m}^{n} a_i \beta^i\ = \sum_{i=-m}^{n} \sum_{j=0}^d \alpha_{ij}\beta^{i+j}\,.
\end{align*}
Suppose for contradiction that $\beta$ is transcendental. 
Therefore, the corresponding integer coefficients of powers of $\beta$ 
on the left hand side and on the right hand side must be equal, particularly
$$
\sum_{\substack{i+j=0 \\0\leq j\leq d \,, -m\leq i \leq n}} \alpha_{ij} =N
$$
This is a contradiction, since the left hand side is bounded by $(d+1) \cdot C$, whereas $N$ can be arbitrarily large.
\end{proof}

\begin{corollary}\label{cor:betaMustBeAlgebraic}
	Let $\beta$ be a complex number such that $|\beta|>1$ and let $\A\subset \ZZ[\beta]$ be a finite alphabet 
	with $0\in\A$ and  $1\in\fin{\A}$, resp.\ $1\in\A[\beta]$. 
	If the set $\fin{\A}$, resp.\ $\A[\beta]$, is closed under addition, then $\beta$ is an algebraic number.
\end{corollary}
\begin{proof}
	The closedness of $\fin{\A}$ under addition and  $1\in\fin{\A}$ implies $\NN\subset\fin{\A}$.
	If $\A[\beta]$ is closed under addition and $1\in\A[\beta]\subset\fin{\A}$, then $\NN\subset\A[\beta]\subset\fin{\A}$.
	In both cases, Lemma~\ref{lem:betaMustBeAlgebraic} applies.
\end{proof}

The concept of parallelism for operations on representations is formalized by the following definition.
\begin{definition}
Let $\A$ and $\B$ be alphabets. A function $\varphi:\B^\ZZ \rightarrow \A^\ZZ$ is said to be \emph{$p$-local} if there exist $r,t\in\NN$ satisfying $p=r+t+1$ and a function $\phi: \B^p \rightarrow \A$ such that, for any $w=(w_j)_{j\in\ZZ}\in\B^\ZZ$ and its image $z=\varphi(w)=(z_j)_{j\in\ZZ}\in\A^\ZZ$, we have $z_j=\phi(w_{j+t},\cdots,w_{j-r})$ for every $j\in\ZZ$. The parameter $t$, resp. $r$, is called \emph{anticipation}, resp. \emph{memory}.
\end{definition}

In other words, every digit can by determined from only limited number of neighboring input digits.
Since a $(\beta,\A+\A)$-representation of sum of two numbers can be easily obtained by digit-wise addition,
the crucial part of parallel addition is conversion from the alphabet $\A+\A$ to $\A$.

\begin{definition}
\label{def:digitSetConversion}
Let $\beta$ be a base and let $\A$ and $\B$ be alphabets containing 0. A function $\varphi:\B^\ZZ\rightarrow \A^\ZZ$ such that
  \begin{enumerate}
      \item for any $w=(w_j)_{j\in\ZZ}\in\B^\ZZ$ with finitely many non-zero digits, $z=\varphi(w)=(z_j)_{j\in\ZZ}\in\A^\ZZ$ has only finite number of non-zero digits, and
      \item $\sum_{j\in\ZZ} w_j \beta^j= \sum_{j\in\ZZ} z_j \beta^j$,
  \end{enumerate}
  is called a \emph{digit set conversion} in the base $\beta$ from $\B$ to $\A$. Such a conversion $\varphi$ is said to be \emph{computable in parallel} if $\varphi$ is a $p$-local function for some $p\in\NN$. \emph{Parallel addition} in a numeration system $(\beta,\A)$ is a digit set conversion in the base $\beta$ from $\A+\A$ to $\A$, which is computable in parallel.
\end{definition}

\section{\texorpdfstring{Necessary conditions on alphabets allowing parallel addition in $\A\subset\Zbeta$}{Necessary conditions on alphabets allowing parallel addition in A subset Z[beta]}}
\label{sec:necessaryAlphabet}
Through this section, we assume that the base $\beta$ is an algebraic number and the alphabet~$\A$ is a finite subset of $\Zbeta$ such that $\{0\}\subsetneq \A$.
The finiteness of the alphabet is a natural assumption for a practical numeration system, whereas the requirement that $\beta$ is an algebraic number is justified by  Corollary~\ref{cor:betaMustBeAlgebraic}.

We recall that for an algebraic number $\beta$, if $\alpha, \gamma,\delta$ are elements of $\ZZ[\beta]$, then \emph{$\gamma$ is congruent to $\delta$ modulo $\alpha$ in $\ZZ[\beta]$}, denoted by $\gamma\equiv_\alpha\delta$, if there exists $\varepsilon\in\Zbeta$ such that $\gamma-\delta=\alpha\varepsilon$.

In this section, we recall the known results on necessary properties of integer alphabets allowing parallel addition, and we extend them to non-integer alphabets. In~\cite{minAlph}, the following statement is proven:

\begin{theorem} \label{thm:reprBetaMinusOne}
	Let $(\beta,\A)$ be a numeration system such that $\A\subset\Zbeta$.
	If there exists a $p$-local parallel addition in $(\beta,\A)$ defined by a function $\phi: (\A+\A)^p \rightarrow \A$,
	then $\phi(b,\dots,b)\equiv_{\beta-1} b$ for any $b \in \A+\A$.
\end{theorem}

The same paper explain that, when considering only integer alphabets $\A \subset \ZZ$ from the perspective of parallel addition algorithms, all the numbers $\beta$, $1/\beta$, and their algebraic conjugates behave analogously: parallel addition algorithms exist either for all, or for none of them. This statement can be extended to non-integer alphabets as well -- the following lemma summarizes that if we have a parallel addition algorithm for a base $\beta$, then we easily obtain such an algorithm also for conjugates of $\beta$ by field isomorphism. Regarding the base $1/\beta$, we can use the equality $\fin{\A} = \fin[1/\beta]{\A}$ to transfer the parallel addition algorithm, and thus in fact drop the requirement on the base to be greater than~$1$ in modulus.

\begin{lemma}
\label{lem:parAddAlgForConjugate}
	Let $(\beta,\A)$ be a numeration system such that $\A\subset\Zbeta$ and $\beta$ is an algebraic number.
	Let $\beta'$ be a conjugate of $\beta$ such that $|\beta'|\neq 1$ and $\sigma:\QQ(\beta)\mapsto \QQ(\beta')$ be the corresponding field isomorphism.
	If there is a $p$-local parallel addition function $\varphi$ in $(\beta,\A)$, 
	then there exists a $p$-local parallel addition function~$\varphi'$ in~$(\beta',\A')$, where $\A'=\{\sigma(a) \colon a\in\A\}$.
\end{lemma}

\begin{proof}
Let $\phi:\A^p\rightarrow\A$ be a mapping which defines $\varphi$ with $p=r+t+1$. We define a mapping $\phi':(\A')^p\rightarrow \A'$ by
$$
\phi'(w'_{j+t}, \dots, w'_{j-r})=\sigma\left(\phi\left(\sigma^{-1}(w'_{j+t}), \dots, \sigma^{-1}(w'_{j-r})\right)\right)\,.
$$
Next, we define a digit set conversion  $\varphi':(\A'+\A')\rightarrow\A'$ by $\varphi'(w')=(z'_j)_{j\in\ZZ}$, where $w'=(w'_j)_{j\in\ZZ}$ and $z'_j=\phi'(w'_{j+t}, \dots, w'_{j-r})$. Obviously, if $w'$ has only finitely many non-zero entries, then there is only finitely many non-zeros in $(z'_j)_{j\in\ZZ}$, since
\begin{align*}
	\phi'(0, \dots, 0)&=\sigma\left(\phi\left(\sigma^{-1}(0), \dots, \sigma^{-1}(0)\right)\right) \\
	&=\sigma\left(\phi\left(0, \dots, 0\right)\right)=\sigma\left(0\right)=0\,.
\end{align*}
The value of the number represented by $w'$ is also preserved:
\begin{align*}
\sum_{j\in\ZZ}w'_j {\beta'}^j&=\sum_{j\in\ZZ}\sigma(w_j) \sigma(\beta)^j =\sigma\left(\sum_{j\in\ZZ}w_j\beta^j\right) \\
	&=\sigma\left(\sum_{j\in\ZZ}z_j\beta^j \right) \\
 & =\sigma\left(\sum_{j\in\ZZ}\phi\left(w_{j+t}, \dots,w_{j-r}\right)\beta^j\right) \\
    &=\sum_{j\in\ZZ}\sigma(\phi\left(w_{j+t}, \dots,w_{j-r}\right)){\beta'}^j=\sum_{j\in\ZZ}z'_j {\beta'}^j \,,
\end{align*}
where $w_j=\sigma^{-1}(w'_j)$ for $j\in\ZZ$ and $\varphi((w_j)_{j\in\ZZ})=(z_j)_{j\in\ZZ}$.
\end{proof}

Next, it is shown again in \cite{minAlph} that if a base $\beta$ has a real conjugate greater than one, then there are some extra requirements on the alphabet $\A\subset\Zbeta$.
The following lemma strengthens the results a bit.

\begin{lemma}
\label{lem:alphabetRestrictions}
Let $(\beta,\A)$ be a numeration system such that $\A\subset\Zbeta$ and $1<\beta\in\RR$. Let $\lambda=\min \A$ and $\Lambda=\max \A$. If there exists a $p$-local parallel addition in $(\beta, \A)$, with $p=r+t+1$, defined by a mapping $\phi\colon (\A+\A)^p\rightarrow \A$, then:
\begin{enumerate}
	\item $\phi(b,\dots,b)\neq \lambda$ for all $b\in\A+\A$ such that ${b>\lambda{} \land (b\geq 0 \lor t=0)}$, 
	\item $\phi(b,\dots,b)\neq \Lambda$ for all $b\in\A+\A$ such that ${b<\Lambda{} \land (b\leq 0 \lor t=0)}$, 
	\item If $\Lambda\neq 0$, then $\phi(\Lambda,\dots,\Lambda)\neq \Lambda$,
	\item If $\lambda\neq 0$, then $\phi(\lambda,\dots,\lambda)\neq \lambda$.
\end{enumerate}
\end{lemma}
\begin{proof}
(1.): Let $b\in\A+\A$ be such that $b>\lambda$. Assume, for contradiction, that $\phi(b,\dots,b)= \lambda$. 
We follow the proof of Claim 3.5. in \cite{minAlph}.
For any $n\in\NN, n\geq 1$, we consider the number represented by
\begin{equation*}
	^{\omega}\!0 \underbrace{b\dots b}_{t} \underbrace{b\dots b}_{n}\bullet \underbrace{b\dots b}_{r}0^\omega\,.
\end{equation*}
Its representation after the digit set conversion has the form
\begin{equation*}
	^{\omega}\!0 \underbrace{w_{r+t}\dots w_{1}}_{\beta^n W}\underbrace{\lambda\dots \lambda}_{n}\bullet \widetilde{w_1}\dots \widetilde{w_{r+t}}0^\omega\,,
\end{equation*}
where $W=\sum_{j=1}^{r+t} w_j \beta^{j-1}$ and $w_j,\tilde{w}_j\in\A$.
Since both representations have the same value, we get:
\begin{equation*}
	b \sum_{j=-r}^{n+t-1} \beta^j =  \beta^n W + \lambda \sum_{j=0}^{n-1} \beta^j + \sum_{j=1}^{r+t}\widetilde{w_j} \beta^{-j} 
\end{equation*}
for all $n\geq 1$. Corollary 3.6. in  \cite{minAlph} gives that $W=\frac{b \beta^t-\lambda}{\beta-1}$. Thus,
\begin{align*}
b \sum_{j=-r}^{-1} \beta^j &+ b \frac{\beta^{n+t}-1}{\beta-1} =\\
&=\beta^n\frac{b \beta^t-\lambda}{\beta-1}+\lambda\frac{\beta^n-1}{\beta-1}+ \sum_{j=1}^{t+r} \widetilde{w_j}\beta^{-j}\,.
\end{align*}
Hence
\begin{align*}
	b&\left(\sum_{j=1}^{r} \frac{1}{\beta^j} +\frac{-1}{\beta-1}\right)=\lambda \frac{-1}{\beta-1}+\sum_{j=1}^{t+r} \widetilde{w_j}\frac{1}{\beta^{j}} \\
	&\geq \lambda \left(\frac{-1}{\beta-1}+ \sum_{j=1}^{t+r}\frac{1}{\beta^{j}}\right)\,.
\end{align*}
Using $\frac{1}{\beta-1}=\sum_{j=1}^{\infty}\frac{1}{\beta^{j}}$, we get 
\begin{align*}
	-b\frac{1}{\beta^{r}}\frac{1}{\beta-1}&=-b\sum_{j=r+1}^{\infty} \frac{1}{\beta^j} \geq -\lambda \sum_{j=r+t+1}^\infty \frac{1}{\beta^{j}} \\
	 &= -\lambda \frac{1}{\beta^{r+t}}\frac{1}{\beta-1} \,.
\end{align*}
Thus, we have $ \lambda \geq b \beta^t$. If $t=0$, then it contradicts the assumption $b>\lambda$.
If $b\geq 0$, then $ \lambda \geq b \beta^t \geq b$ since $\beta>1$, which is also a contradiction.

The proof of (2.) is similar.
For (3.) and (4.), see \cite{minAlph}.
\end{proof}

Now we can conclude with the following statement.

\begin{theorem}
\label{thm:betaPositiveImpliesTwoMoreElements}
Let $(\beta,\A)$ be a numeration system such that $\A\subset\Zbeta$, $\beta$ is an algebraic number with a positive real conjugate and there is parallel addition in $(\beta,\A)$. Let $\lambda=\min \A$ and $\Lambda=\max \A$. If $\lambda\equiv_{\beta-1}\Lambda$, then there exists $c\in\A, \lambda\neq c\neq \Lambda$ such that $\lambda\equiv_{\beta-1} c \equiv_{\beta-1} \Lambda$. If  $\lambda\not\equiv_{\beta-1}\Lambda$, then there exist $a,b\in\A, a\neq \lambda, b \neq \Lambda$ such that $a\equiv_{\beta-1} \lambda$ and $b\equiv_{\beta-1} \Lambda$.
\end{theorem}

\begin{proof}
By Lemma~\ref{lem:parAddAlgForConjugate}, we can assume that the base $\beta$ itself is real and greater than one.\\
Let $\phi$ be a mapping which defines the parallel addition. Since $\{0\}\subsetneq \A$, we know that $\Lambda>0$ or $\lambda<0$. Assume that $\Lambda>0$, the latter one is analogous. By Theorem~\ref{thm:reprBetaMinusOne} and Lemma~\ref{lem:alphabetRestrictions}, $\Lambda\equiv_{\beta-1}\phi(\Lambda,\dots,\Lambda)\in\A$ and $\lambda \neq \phi(\Lambda,\dots,\Lambda)\neq \Lambda$. Hence, $\phi(\Lambda,\dots,\Lambda)$ is a digit of $\A$ which belongs to the same congruence class as $\Lambda$.

If $\lambda\equiv_{\beta-1}\Lambda$, the claim follows, with $c = \phi(\Lambda, \ldots, \Lambda)$.

The case that $\lambda\not\equiv_{\beta-1}\Lambda$ is divided into two sub-cases. If $\lambda\neq 0$, then we have $\lambda\equiv_{\beta-1}\phi(\lambda,\dots,\lambda)\in\A$ and  $\phi(\lambda,\dots,\lambda)\neq \lambda$ again by Theorem~\ref{thm:reprBetaMinusOne} and Lemma~\ref{lem:alphabetRestrictions}, which implies the statement, with $a = \phi(\lambda, \ldots, \lambda)$ and $b = \phi(\Lambda, \ldots, \Lambda)$.

If $\lambda=0$, then all elements of $\A+\Lambda$ are positive. Suppose, for contradiction, that there is no nonzero element of the alphabet~$\A$ congruent to~$0$ modulo $\beta-1$.
Let $k$ be the number of congruence classes occurring in $\A$ and let $R$ be a subset of $\A$ such that there is exactly one representative of each of those $k$ congruence classes.
 For $d\in\Lambda+R$, the value $\phi(d,\dots,d)\in\A$ is not congruent to~$0$, as $\phi(d,\dots,d)\neq\lambda=0$ by Lemma~\ref{lem:alphabetRestrictions} and the congruence class containing zero has only one element in $\A$, by the previous assumption. Therefore, the values $f_j = \phi(d_j,\dots,d_j)\in\A$ for $k$ distinct digits $d_j = \Lambda+e_j \in \Lambda+R$ belong to only $k-1$ congruence classes modulo $\beta-1$.
Hence, there exist two distinct elements $d_1, d_2 \in \Lambda+R$ such that $f_1 \equiv_{\beta-1} f_2$. Due to
$$
f_j = \phi(d_j, \ldots, d_j) \equiv_{\beta-1} d_j = \Lambda + e_j \mathrm{\ for\ each\ } j \, ,
$$
we obtain also $e_1 \equiv_{\beta-1} e_2$ which contradicts the construction of the set $R$.

\end{proof}

%%%%%%%%%%%%%%%%%%%%%%%%%%%%%%%%%%%%%%%%%%%%%%%%%%%%%%%%%%%%%%%%%%%%%%%%%%%%%%%%%%%%%%%%%%%%%%%%%%%

\subsection{\texorpdfstring{$\A[\beta]$ closed under addition}{A[beta] closed under addition}}

In order to express the properties of the alphabet $\A$ allowing parallel addition in terms of representatives modulo~$\beta$ and $\beta-1$, we restrict ourselves to alphabets such that $\A[\beta]$ is closed under addition, or even slightly stronger condition that $\A[\beta]=\ZZ[\beta]$. The following theorem summarizes some consequences of these assumptions.

\begin{theorem}
\label{thm:betaMustBeExpanding}
Let $(\beta, \A)$ be a numeration system such that $\A\subset\Zbeta$ and $1\in\A[\beta]$. The following statements hold:
\begin{enumerate}
	\item If $\A[\beta]$ is closed under addition, then $\NN[\beta]\subset\A[\beta]$.
	\item $\A[\beta]$ is additive Abelian group if and only if $\A[\beta]=\ZZ[\beta]$.
	\item If $\NN\subset\A[\beta]$, then $\beta$ is expanding, i.e., $\beta$ is an algebraic number with all conjugates greater than~1 in modulus.
\end{enumerate}
\end{theorem}

\begin{proof}
(1.): Obviously, if $\A[\beta]$ is closed under addition, then $\NN\subset \A[\beta]$. Since $0\in\A$ by our general assumption, also $\beta\cdot\A[\beta]\subset \A[\beta]$. Therefore, $\beta\cdot\NN\subset \A[\beta]$ and the claim $\NN[\beta]\subset\A[\beta]$ follows by induction.

(2.): The assumption that $\A[\beta]$ is closed also under subtraction and (1.) imply that $\Zbeta=\NN[\beta]-\NN[\beta]\subset\A[\beta]$, and obviously $\A[\beta]\subset\Zbeta$. The opposite implication is trivial.

(3.): By Lemma~\ref{lem:betaMustBeAlgebraic}, $\beta$ is an algebraic number since $\A[\beta]\subset\fin{\A}$. The proof that $\beta$ must be expanding is based on the paper of S. Akiyama and T. Zäimi \cite{Akiyama20131616}.
    Let $\beta'$ be an algebraic conjugate of $\beta$ and  $\sigma: \QQ(\beta)\rightarrow \QQ(\beta')$ be the field isomorphism such that $\sigma(\beta)=\beta'$. 
	Since $\NN\subset \A[\beta]$, for all $n \in \NN$ there exist $a_0, \dots , a_N\in\A$ such that
    $$
    \sum_{i=0}^{N}a_i\beta^i =n=\sigma(n)=\sum_{i=0}^{N}\sigma(a_i)(\beta')^i\,.
    $$
    Denoting $\tilde m:= \max\{|\sigma(a)|\colon a\in\A\}$, we have
    \begin{align*}
	    n&=|n|\leq\sum_{i=0}^{N}|\sigma(a_i)|\cdot|\beta'|^i \\
	    &\leq \sum_{i=0}^{\infty}|\sigma(a_i)|\cdot|\beta'|^i \leq \tilde m\sum_{i=0}^{\infty}|\beta'|^i\,.    
    \end{align*}
    As $n$ is arbitrarily large, the sum on the right  side diverges, which implies that $|\beta'|\geq 1$. Thus, all conjugates of $\beta$ are at least one in modulus.

    If the degree of $\beta$ is one, the statement is obvious.  Therefore, we may assume that $\deg \beta \geq 2$.

    Suppose  for contradiction that $|\beta'|=1$ for an algebraic conjugate $\beta'$  of $\beta$. The complex conjugate $\overline{\beta'}$ is also an algebraic conjugate of $\beta$. Take any algebraic conjugate $\gamma$ of $\beta$ and the isomorphism $\sigma': \QQ(\beta')\rightarrow \QQ(\gamma)$ given by $\sigma'(\beta')=\gamma$.
    Now
    \begin{align*}
		    \frac{1}{\gamma}&=\frac{1}{\sigma'(\beta')}=\sigma'\left(\frac{1}{\beta'}\right)=\sigma'\left(\frac{\overline{\beta'}}{\beta'\overline{\beta'}}\right) \\
		    &=\sigma'\left(\frac{\overline{\beta'}}{|\beta'|^2}\right)=\sigma'(\overline{\beta'})\,.    
    \end{align*}
    Hence, $\frac{1}{\gamma}$ is also an algebraic conjugate of $\beta$. Moreover, $\left|\frac{1}{\gamma}\right|\geq 1$ and $|\gamma|\geq 1$, which implies that $|\gamma|=1$. 
    We may choose $\gamma=\beta$, which contradicts $|\beta|>1$. Thus all conjugates of $\beta$ are greater than one in modulus, i.e., $\beta$ is an expanding algebraic number.
\end{proof}

Let us remark that the assumption on $\A[\beta]$ to be closed under addition is satisfied by a wide class of numeration systems.
Namely, if a numeration system $(\beta,\A)$ allows $p$-local parallel addition such that $p=r+1$, i.e.,
there is no anticipation, then $\A[\beta]$ is obviously closed under addition.
Hence, (1.) and (3.) give the following corollary.
\begin{corollary}
	\label{cor:noAnticipationGivesExpanding}
	Let $(\beta, \A)$ be a numeration system such that $1\in\A[\beta]$ and $\A\subset\ZZ[\beta]$. 
	If $(\beta, \A)$ allows parallel addition without anticipation, then $\beta$ is expanding.
\end{corollary}
For $\beta$ expanding, Lemma 8 in \cite{akiyama2012} provides a so called weak representation of zero property such that the absolute term is dominant,
and hence parallel addition in the base $\beta$ without anticipation is obtained for some integer alphabet $\A_{int}$ according Theorem 4.3. in \cite{parAddNS}.

In what follows, we assume $\A[\beta]=\ZZ[\beta]$, although the weaker assumption, $\A[\beta]$ being closed under addition, would be sufficient. The reason is that subtraction is also required in applications using parallel addition, such as on-line multiplication and division. Hence the assumption is justified by (2.) of Theorem~\ref{thm:betaMustBeExpanding}. Let us  mention that, for instance, the numeration system $(2,\{0,1,2\})$ allows parallel addition, but $(2,\{0,1,2\})$ is obviously not closed under subtraction.

\begin{theorem}
If a numeration system $(\beta, \A)$ with $\A[\beta]=\ZZ[\beta]$ allows parallel addition, then the alphabet $\A$ contains at least one representative of each congruence class modulo~ $\beta$ and modulo $\beta-1$ in $\Zbeta$.
\label{thm:representativesInAlphabet}
\end{theorem}

\begin{proof}
Let $x=\sum_{i=0}^N x_i \beta^i$ be an element of $\Zbeta$. Since $x_0\in \ZZ \subset \A[\beta]$, we have
$$
x\equiv_\beta x_0=\sum_{i=0}^{n}a_i\beta^i \equiv_\beta a_0 \,, \mathrm{\ where\ } a_i\in \A.
$$
Hence, for any element $x \in \ZZ[\beta]$, there is a digit $a_0\in\A$ such that $x\equiv_\beta a_0$.

In order to prove that there is an element of $\A$ congruent to $x$ modulo $\beta-1$, we use the binomial theorem:
$$
x=\sum_{i=0}^N x_i \beta^i=\sum_{i=0}^N x_i (\beta-1+1)^i=\sum_{i=0}^N x'_j (\beta-1)^j\,,
$$
for some $x'_j\in\ZZ$. Hence
$$
x\equiv_{\beta-1} x'_0=\sum_{i=0}^{n}a_i\beta^i \, ,
$$
for some $a_i\in \A$. We prove by induction with respect to $n$ that $x'_0\equiv_{\beta-1} a$ for some $a\in\A$.
If $n=0$, then $x'_0=a_0 \in \A$. For $n+1$, we have
\begin{align*}
x'_0&=\sum_{i=0}^{n+1}a_i\beta^i    =a_0 + (\beta-1)\sum_{i=0}^{n}a_{i+1}\beta^i+ \sum_{i=0}^{n}a_{i+1}\beta^i \\
&\equiv_{\beta-1}a_0 +a' \equiv_{\beta-1}a \in\A\,,
\end{align*}
where we use the induction assumption $\sum_{i=0}^{n}a_{i+1}\beta^i\equiv_{\beta-1} a'\in\A$ and the statement of Theorem~\ref{thm:reprBetaMinusOne}, i.e, for each digit $b \in\A+\A$ there is a digit  $a\in\A$ such that $b \equiv_{\beta-1} a$.
\end{proof}

\subsection{\texorpdfstring{Lower bound on $\#\A$}{Lower bound on A}}

When deriving the minimal size of alphabets for parallel addition, we assume that the base $\beta$ is an algebraic integer (in this whole subsection), since it enables us to count the number of congruence classes, and hence to provide an explicit lower bound on the size of alphabet allowing parallel addition. In what follows, the monic minimal polynomial of an algebraic integer $\alpha$ is denoted by $m_\alpha$.

Let $d$ be the degree of $\beta$. It is well known that $\Zbeta=\{\sum_{i=0}^{d-1} x_i \beta^i\colon x_i\in\ZZ\}$ if and only if $\beta$ is an algebraic integer. Hence, there is an obvious bijection $\pi:\Zbeta \rightarrow \ZZ^{d}$ given by
\begin{equation*}
    \pi(u)=\vect{u} 
\end{equation*}
for every $u=\sum_{i=0}^{d-1} u_i \beta^i \in \Zbeta$.
Moreover, the additive group $\ZZ^d$ can be equipped with a multiplication such that $\pi$ is a ring isomorphism. In order to do that, we recall the concept of companion matrix.

\begin{definition}
Let  $p(x)=x^d +p_{d-1}x^{d-1}+ \cdots + p_1 x+p_0 \in \ZZ[x]$ be a monic polynomial with integer coefficients, $d\geq 1$. The matrix
\begin{equation*}
    S := \begin{pmatrix}
            0 & 0 & \cdots & 0 & -p_0 \\
            1 & 0 & \cdots & 0 & -p_1 \\
            0 & 1 & \cdots & 0 & -p_2 \\
            \vdots &   & \ddots & & \vdots \\
            0 & 0 & \cdots & 1 & -p_{d-1}
            \end{pmatrix} \in \ZZ^{d\times d}
\end{equation*}
is the \emph{companion matrix} of the polynomial $p$.
\end{definition}

It is well known (see for instance \cite{horn}) that the characteristic polynomial of the companion matrix $S$ is $p$. The matrix $S$ is also root of the polynomial $p$.

The claim of the following theorem, which provides the required multiplication in~$\ZZ^d$, is discussed in \cite{katai}. These topics are also more elaborated in \cite{vu,dp}.
\begin{theorem}
\label{thm:isomorphismWithZd}
Let $\beta$ be an algebraic integer of degree $d\geq 1$ and let $S$ be the companion matrix of $m_\beta$. If the multiplication $\odot_\beta: \ZZ^d \times \ZZ^d \rightarrow \ZZ^d$ is defined by
\[
u \odot_\beta v := \left(\multMat{u}\right)\cdot v 
\]
for all $ u=\vect{u}, v \in \ZZ^d$, then $(\ZZ^d,+,\odot_\beta)$ is a commutative ring which is isomorphic to $\Zbeta$ by the mapping $\pi$.
\end{theorem}

One of the consequences is the following lemma. Although it is a known result, we include its proof here, to be more self-contained. Let us recall that for a non-singular integer matrix $M\in\ZZ^{d\times d}$, two vectors $x,y \in \ZZ^d$ are \emph{congruent modulo $M$ in $\ZZ^d$}, denoted by $x\equiv_M y$, if $x-y \in M\ZZ^d$.

\begin{lemma}
Let $\beta$ be an algebraic integer of degree $d$ and $\alpha\in\Zbeta$ be such that $\deg \alpha = \deg\beta$. The number of congruence classes modulo $\alpha$ in $\Zbeta$ is $|m_\alpha(0)|$.
\label{lem:numbCongruenceClasses}
\end{lemma}
\begin{proof}
The number $\alpha$ is an algebraic integer, since it is well known that sum and product of algebraic integers is an algebraic integer.
Let $\gamma, \delta\in\Zbeta$ and let $S$ be the companion matrix of the minimal polynomial $m_\beta$ of the algebraic integer $\beta$. Let $\pi(\alpha) =\vect{a}$, with $\alpha = \sum_{i=0}^{d-1} a_i \beta^i$.
If we set $S_\alpha:=\multMat{a}$, then the congruences $\equiv_\alpha$ in $\ZZ[\beta]$ and $\equiv_{S_\alpha}$ in $\ZZ^d$ fulfill:
\begin{align*}
\gamma\equiv_\alpha \delta &\iff \exists \varepsilon \in \Zbeta \colon \gamma-\delta= \alpha \varepsilon \\
&\iff \exists z=\pi(\varepsilon) \in \ZZ^d \colon \pi(\gamma)-\pi(\delta)= \\
& \qquad \quad= \pi(\gamma-\delta)= \pi(\alpha)\odot_\beta z = S_\alpha \cdot z \\
&\iff \pi(\gamma) \equiv_{S_\alpha} \pi(\delta)\,.
\end{align*}
Thus, the number of congruence classes modulo $\alpha$ in $\Zbeta$ equals the number of congruence classes modulo $S_{\alpha}$ in $\ZZ^d$, which is known to be $|\det S_{\alpha}|$.

To show that $|\det S_{\alpha}|=|m_\alpha(0)|$, we proceed in the following way: the characteristic polynomial of the companion matrix $S$ is the same as the minimal polynomial of $\beta$.
Since minimal polynomials have no multiple roots, $S$ is diagonalizable over $\CC$, i.e., $S=P^{-1}DP$ where $D$ is a diagonal matrix with the conjugates of $\beta$ on the diagonal, and $P$ is a non-singular complex matrix. The matrix $S_\alpha$ is also diagonalized by $P$:
\begin{align*}
	S_\alpha&=\sum_{i=0}^{d-1} a_i S^i= \sum_{i=0}^{d-1} a_i \left(P^{-1}DP\right)^i \\
	&=P^{-1}\underbrace{\left(\sum_{i=0}^{d-1} a_i D^i\right)}_{D_\alpha}P\,.
\end{align*}
It is known (see for instance \cite{chapman}) that if $\sigma:\QQ(\beta)\rightarrow\QQ(\beta')$ is a field isomorphism and $\alpha\in\QQ(\beta)$, then $\sigma(\alpha)$ is a conjugate of $\alpha$, and we obtain all conjugates of $\alpha$ in this way.
Since $\alpha=\sum_{i=0}^{d-1} a_i \beta^i$, $\deg \alpha = \deg\beta$ and  $D$ has conjugates of $\beta$ on the diagonal, the diagonal elements of the diagonal matrix $D_\alpha$ are precisely all conjugates of $\alpha$.
Hence, $|\det S_\alpha|$ equals absolute value of the product of all conjugates of $\alpha$, which is $|m_\alpha(0)|$.
\end{proof}

%%%%%%%%%%%%%%%%%%%%%%%%%%%%%%%%%%%%%%%%%%%%%%%%%%%%%%%%%%%%%%%%%%%%%%%%%%%%%%%%%%%%%%%%%%%%%%%%%%%

Finally, we put together the fact that the alphabet~$\A$ for parallel addition in base~$\beta$ contains all representatives modulo~$\beta$ and modulo $\beta-1$, the derived formula for the number of congruence classes, and also specific restrictions on alphabets for parallel addition in a base with some positive real conjugate.

%%%%%%%%%%%%%%%%%%%%%%%%%%%%%%%%%%%%%%%%%%%%%%%%%%%%%%%%%%%%%%%%%%%%%%%%%%%%%%%%%%%%%%%%%%%%%%%%%%%

\begin{theorem}
\label{thm:lowerBoundAlphabet}
Let $(\beta,\A)$ be a numeration system  such that $\beta$ is an algebraic integer and $\A[\beta]=\ZZ[\beta]$. If the numeration system $(\beta, \A)$ allows parallel addition, then
$$
\#\A \geq \max \{|m_\beta(0)|, |m_\beta(1)|\}\,.
$$
Moreover, if $\beta$ has a positive real conjugate, then
$$
\#\A \geq \max \{|m_\beta(0)|, |m_\beta(1)|+2\}\,.
$$
\end{theorem}

\begin{proof}
By Theorem~\ref{thm:representativesInAlphabet}, the alphabet~$\A$ for parallel addition must contain all representatives modulo~$\beta$ and modulo $\beta-1$ in~$\ZZ[\beta]$. The numbers of congruence classes are $|m_\beta(0)|$ and $|m_{\beta-1}(0)|$, respectively, by Lemma~\ref{lem:numbCongruenceClasses}. Obviously, $m_{\beta-1}(x) = m_\beta (x+1)$. Thus $m_{\beta-1}(0) = m_\beta (1)$.

Theorem~\ref{thm:betaPositiveImpliesTwoMoreElements} ensures that if the minimal and maximal element of $\A$ are congruent modulo $\beta-1$, then there are at least three digits of $\A$ in this class. Otherwise, the class of the minimal and also the class of the maximal element of $\A$ have at least two elements. Both lead to the conclusion that $\#\A \geq |m_\beta(1)|+2$.
\end{proof}
We remark that the obtained bound is basically the same as the one for integer alphabets in \cite{minAlph}.

\section{Necessary and sufficient condition on bases for parallel addition}
\label{sec:condBase}
P.~Kornerup \cite{kornerup} proposed a more general concept of parallel addition called $k$-block parallel addition. The idea is that blocks of $k$ digits are considered as one digit in the new numeration system with base being the $k$-th power of the original one.

\begin{definition}
For a positive integer $k$, the numeration system $(\beta, \A)$ allows \emph{$k$-block parallel addition} if there exists parallel addition in $(\beta^k, \A_{(k)})$ , where $\A_{(k)}=\{a_{k-1}\beta^{k-1}+\dots+a_1 \beta+a_0 \colon a_i\in \A\}$.
\end{definition}
We remark that 1-block parallel addition is the same as parallel addition. C.~Frougny, P.~Heller, E.~Pelantov\'a and M.~Svobodov\'a \cite{kBlock} showed that for a given base $\beta$, there exists an integer alphabet $\A$ such that $(\beta, \A)$ allows parallel addition if and only if $\beta$ is an algebraic number with no conjugates of modulus~$1$. Moreover, it was shown that the concept of $k$-block parallel addition does not enlarge the class of basis allowing parallel addition in case of integer alphabets.
We prove an extension of these statements also to alphabets being subsets of $\Zbeta$ in Theorem~\ref{thm:kblock}.
Although the $k$-block concept does not enlarge the class of bases for parallel addition, it might decrease the minimal size of the alphabet.

\begin{theorem}
\label{thm:kblock}
Let $\beta$ be a complex number such that $|\beta|>1$. There exists an alphabet $\A \subset \ZZ[\beta]$ with $0\in \A$ and $1\in\fin{\A}$ which allows $k$-block parallel addition in $(\beta,\A)$ for some $k \in \NN$, if and only if $\beta$ is an algebraic number with no conjugate of modulus~$1$. If this is the case, then there also exists an alphabet in $\ZZ$ allowing 1-block parallel addition in base~$\beta$.
\end{theorem}

\begin{proof}
If the base $\beta$ is an algebraic number with no conjugates of modulus~1, then \cite{parAddNS} provides $a\in\NN$ such that the alphabet $\A=\{-a, -a-1, \dots, 0, \dots, a-1, a\}$ allows 1-block parallel addition. Obviously, $0\in\A\subset\ZZ$ and $1\in\fin{\A}$.

For the opposite implication, $\beta$ is an algebraic number by Corollary~\ref{cor:betaMustBeAlgebraic}. Let $r,s\in\NN$ and $u_{-r}, \dots, u_{s}\in \A$ be such that
\begin{equation}\label{jednicka}
    \sum_{j=-r}^s u_j \beta^j=1\in \fin{\A} \, .
\end{equation}

Now we slightly modify the proof from \cite{kBlock} to show that if $\beta$ has a conjugate of modulus~1, then there is no alphabet in $\Zbeta$ allowing block parallel addition.
Let $\gamma$ be a conjugate of $\beta$ such that $|\gamma|\,=1$ and let $\sigma:\QQ(\beta)\rightarrow \QQ(\gamma)$ be the field isomorphism such that $\sigma(\beta)=\gamma$. Let $\A':=\{\sigma(a) \colon a\in\A\}$.
Assume, for contradiction, that there are $k,p\in\NN$ such that there exists $p$-local function performing $k$-block parallel addition on $(\beta, \A)$. We denote
$$
S:=\max\left\{\left|\sum_{j=0}^{pk-1} a_j  \gamma^j\right| \colon a_j \in \A'\right\}\,.
$$
Since there are infinitely many $j$ such that $\Re{\gamma^j} >\frac{1}{2}$, there exists $N>p$ and indices $0\leq j_1<\dots<j_m\leq kN-1$ satisfying $j_{i+1}-j_i>r+s$ for all $i\in\{1,\dots, m-1\}$ such that
$$
2 S<\Re{\sum_{j=0}^{kN-1} \varepsilon_j  \gamma^j } \leq \left|\sum_{j=0}^{kN-1} \varepsilon_j  \gamma^j \right|\,,
$$
where $\varepsilon_j=1$ if $j=j_i$ for some $i\in\{1, \dots,m\}$ and $\varepsilon_j=0$ otherwise. By using the representation \eqref{jednicka} of~1 and the fact that $j_{i+1}-j_i>r+s$, we have $\sum_{j=0}^{kN-1} \varepsilon_j  \gamma^j=\sum_{i=1}^{m} 1\cdot \gamma^{j_i}= \sum_{j=-r}^{kN-1+s} v'_j \gamma^j$ for some $v'_j\in\A'$. Hence
$$
2 S<T:=\max \left\{\left|\sum_{j=-r}^{kN-1+s} a_j  \gamma^j\right| \colon a_j \in \A'\right\}\,.
$$
Let $x'=\sum_{j=-r}^{kN-1+s} \sigma(x_j)  \gamma^j$, where $x_{-r}, \dots, x_{kN-1+s}\in \A$, be such that $|x'|=T$. Let $x=\sum_{j=-r}^{kN-1+s} x_j  \beta^j$, i.e., $x'=\sigma(x)$.
Since there is a $k$-block parallel addition in $(\beta,\A)$, we have
$$
x+x=\sum_{j=kN+s}^{k(N+p)-1+s} z_j  \beta^j +\sum_{j=-r}^{kN-1+s} z_j  \beta^j +\sum_{j=-kp-r}^{-r-1} z_j  \beta^j\,,
$$
where $z_j\in\A$. We denote $z'_j:=\sigma(z_j)$. Hence, $z'_j \in \A'$, and
\begin{align*}
    &|x'| + 2 S   < |x'|+|x'|= |x'+x'| \\
        &   \leq \left|\sum_{j=kN+s}^{k(N+p)-1+s} z'_j  \gamma^j\right| +\left|\sum_{j=-r}^{kN-1+s} z'_j  \gamma^j\right| +\left|\sum_{j=-kp-r}^{-r-1} z'_j  \gamma^j\right| \\
        &   \leq |\gamma^{kN+s}|S+|x'|+|\gamma^{-kp-r}|S= 2S +|x'|\,,
\end{align*}
which is a contradiction.
\end{proof}

%%%%%%%%%%%%%%%%%%%%%%%%%%%%%%%%%%%%%%%%%%%%%%%%%%%%%%%%%%%%%%%%%%%%%%%%%%%%%%%%%%%%%%%%%%%%%%%%%%%
%%%%%%%%%%%%%%%%%%%%%%%%%%%%%%%%%%%%%%%%%%%%%%%%%%%%%%%%%%%%%%%%%%%%%%%%%%%%%%%%%%%%%%%%%%%%%%%%%%%

\section{Conclusion}

We have shown that the necessary conditions on $\beta$ and $\A$ allowing parallel addition that were known for alphabets consisting of consecutive integers can be largely extended to alphabets $\A$ being subsets of $\Zbeta$.

During our investigation, we also considered even more general case: $\beta\in\Zomega$ and $\A\subset\Zomega$, where $\omega$ is an algebraic number.
Clearly, $\Zbeta\subset\Zomega$, but if $\Zbeta\subsetneq\Zomega$, then congruences modulo $\beta$ or $\beta-1$ behave differently in $\Zomega$ than in $\Zbeta$.
Due to this fact, Theorem~\ref{thm:representativesInAlphabet} does not hold in the $\Zomega$ setting, see a counterexample: let $\omega=\frac{1}{2}\sqrt{5}+\frac{1}{2}$, $\beta=-2\omega+1=-\sqrt{5}$ and $\A=\{-2, -1,0,1,2,3\}\subset\Zomega$. This numeration system allows parallel addition, but $-2\equiv_{\beta-1} 0\equiv_{\beta-1}2$ and $-1\equiv_{\beta-1}1\equiv_{\beta-1}3$, with $\equiv_{\beta-1}$ in $\Zomega$. The minimal polynomial of $\beta$ is $x^2-5$, thus there are six congruence classes modulo $\beta-1$ in $\Zomega$, i.e., $\A$ does not contain representatives of all congruence classes. See \cite{dp} for further elaboration.

We conjecture the converse of Corollary~\ref{cor:noAnticipationGivesExpanding}, that is: 
let $(\beta, \A)$ be a numeration system such that $1\in\A[\beta]$ and $\A\subset\ZZ[\beta]$. Let $(\beta,\A)$ allow parallel addition by $p$-local function with $p=r+t+1$.
If $\beta$ is expanding, then the parallel addition is without anticipation, i.e., $t=0$.
%%%%%%%%%%%%%%%%%%%%%%%%%%%%%%%%%%%%%%%%%%%%%%%%%%%%%%%%%%%%%%%%%%%%%%%%%%%%%%%%%%%%%%%%%%%%%%%%%%%

\section{Acknowledgments}
This work was supported by GA\v CR 13-03538S and  SGS 17/193/OHK4/3T/14. The author thanks to Milena Svobodov\'a and Edita Pelantov\'a for fruitful discussions.

%%%%%%%%%%%%%%%%%%%%%%%%%%%%%%%%%%%%%%%%%%%%%%%%%%%%%%%%%%%%%%%%%%%%%%%%%%%%%%%%%%%%%%%%%%%%%%%%%%%

\bibliographystyle{actapoly}

\bibliography{literatura}

\begin{thebibliography}{10}
\providecommand{\url}[1]{\texttt{#1}}
\providecommand{\urlprefix}{URL }
\providecommand{\eprint}[2][]{\url{#2}}
\makeatletter
\def\bibdoi{\begingroup\def\do##1{\catcode
  `##112\relax}\do$\do\\\do\_\do\%\do\^\expandafter\endgroup\@bibdoi}
\def\@bibdoi#1{\href{http://dx.doi.org/#1}{\textsc{doi}:#1}}
\makeatother

\bibitem{avizienis}
A.~Avizienis.
\newblock Signed-digit number representations for fast parallel arithmetic.
\newblock \textit{IEEE Trans Comput} \textbf{10}:389--400, 1961.

\bibitem{kornerup}
P.~Kornerup.
\newblock Necessary and sufficient conditions for parallel, constant time
  conversion and addition.
\newblock \textit{Proc 14th IEEE Symp on Comp Arith} pp. 152--155, 1999.

\bibitem{kBlock}
C.~Frougny, P.~Heller, E.~Pelantov\'a, M.~Svobodov\'a.
\newblock k-block parallel addition versus 1-block parallel addition in
  non-standard numeration systems.
\newblock \textit{Theoret Comput Sci} \textbf{543}:52--67, 2014.

\bibitem{minAlph}
C.~Frougny, E.~Pelantov{\'a}, M.~Svobodov{\'a}.
\newblock Minimal digit sets for parallel addition in non-standard numeration
  systems.
\newblock \textit{J Integer Seq} \textbf{16}:36, 2013.

\bibitem{parAddNS}
C.~Frougny, E.~Pelantov\'a, M.~Svobodov\'a.
\newblock Parallel addition in non-standard numeration systems.
\newblock \textit{Theoret Comput Sci} \textbf{412}:5714--5727, 2011.

\bibitem{Brzicova2016}
M.~Brzicov\'a, C.~Frougny, E.~Pelantov\'a, M.~Svobodov\'a.
\newblock {On-line Multiplication and Division in Real and Complex Bases}.
\newblock In \textit{2016 IEEE 23nd Symp. Comput. Arith. (ARITH)}, pp.
  134--141. IEEE, 2016.
\newblock \bibdoi{10.1109/ARITH.2016.13}.

\bibitem{constParAdd}
J.~Legersk{\'{y}}, M.~Svobodov{\'{a}}.
\newblock {Construction of Algorithms for Parallel Addition}, 2018.
\newblock \url{http://arxiv.org/abs/1801.01062}.

\bibitem{Baker2017}
S.~Baker, Z.~Mas{\'{a}}kov{\'{a}}, E.~Pelantov{\'{a}}, T.~V{\'{a}}vra.
\newblock {On periodic representations in non-Pisot bases}.
\newblock \textit{Monatshefte f{\"{u}}r Math} \textbf{184}(1):1--19, 2017.
\newblock \bibdoi{10.1007/s00605-017-1063-9}.

\bibitem{Akiyama20131616}
S.~Akiyama, T.~Zaïmi.
\newblock Comments on the height reducing property.
\newblock \textit{Cent Eur J Math} \textbf{11}:1616--1627, 2013.

\bibitem{akiyama2012}
S.~Akiyama, P.~Drungilas, J.~Jankauskas.
\newblock Height reducing problem on algebraic integers.
\newblock \textit{Funct Approx Comment Math} \textbf{47}:105--119, 2012.
\newblock \bibdoi{10.7169/facm/2012.47.1.9}.

\bibitem{horn}
R.~A. Horn, C.~R. Johnson.
\newblock \textit{Matrix Analysis}.
\newblock Cambridge University Press, 1990.

\bibitem{katai}
I.~Kátai.
\newblock {Generalized number systems in Euclidean spaces}.
\newblock \textit{Math Comput Model} \textbf{38}:883 -- 892, 2003.

\bibitem{vu}
J.~Legersk\'y.
\newblock \textit{Construction of algorithms for parallel addition}.
\newblock Research project, Czech Technical University in Prague, FNSPE, Czech
  Republic, 2015.
\newblock
  \url{http://jan.legersky.cz/pdf/research_project_parallel_addition.pdf}.

\bibitem{dp}
J.~Legersk\'y.
\newblock \textit{Construction of algorithms for parallel addition in
  non-standard numeration systems}.
\newblock {Master Thesis}, Czech Technical University in Prague, FNSPE, Czech
  Republic, 2016.
\newblock \url{http://jan.legersky.cz/pdf/master_thesis_parallel_addition.pdf}.

\bibitem{chapman}
R.~Chapman.
\newblock {Algebraic Number Theory -- summary of notes}.
\newblock \url{http://empslocal.ex.ac.uk/people/staff/rjchapma/notes/ant2.pdf}.
\newblock Accessed: 2017-12-16.

\end{thebibliography}

\end{document}